\documentclass[12pt,oneside]{amsart}

\usepackage{amssymb,latexsym}

\usepackage{hyperref}
\hypersetup{
  colorlinks   = true, 
  urlcolor     = blue, 
  linkcolor    = blue, 
  citecolor   = red 
}
\usepackage{anysize}
\marginsize{2cm}{2cm}{2cm}{2cm}

\usepackage[pdftex]{graphicx}
\usepackage{subcaption}
\usepackage{mathtools}

\newtheorem{theorem}{Theorem}[section]

\newtheorem{lemma}[theorem]{Lemma}

\newtheorem{conjecture}[theorem]{Conjecture}
\newtheorem{corollary}[theorem]{Corollary}

\theoremstyle{definition}
\newtheorem{definition}[theorem]{Definition}

\theoremstyle{remark}
\newtheorem{remark}[theorem]{Remark}
\newtheorem{question}[theorem]{Question}

\numberwithin{equation}{section}

\DeclareMathOperator{\vol}{vol}

\DeclareMathOperator{\conv}{conv}

\renewcommand{\epsilon}{\varepsilon}
\renewcommand{\phi}{\varphi}
\renewcommand{\kappa}{\varkappa}

\begin{document}

\title{Symplectically self-polar polytopes of minimal capacity}

\author{Mark Berezovik}

\address{Mark Berezovik, School of Mathematical Sciences, Tel Aviv University, Israel,  69978}

\email{m.berezovik@gmail.com}

\thanks{The author was supported by Horizon Europe ERC Grant number 101045750 (HodgeGeoComb) and partially supported by the ISF grant No. 938/22}


\begin{abstract}
 In this paper we continue the study of symplectically self-polar convex bodies started in~\cite{berezovik2022symplectic}. We construct symplectically self-polar convex bodies of the minimal Ekeland--Hofer--Zehnder capacity. This in turn proves that the lower bound for the Ekeland--Hofer--Zehnder capacity for centrally symmetric convex bodies obtained in~\cite{akopyan2017estimating} cannot be improved. We also make some numerical experiments and speculations regarding the minimal volume of symplectically self-polar convex bodies.
\end{abstract}

\maketitle

\section{Introduction}

	Consider the standard symplectic form $\omega$ in $\mathbb{R}^{2n}$. For a convex set $X \subset \mathbb{R}^{2n}$ define its \emph{symplectically polar convex set} as
	\[
		X^{\omega} = \{y \in \mathbb{R}^{2n}\ |\ \forall x \in X\  \omega(x,y) \leq 1\}.		
	\] 
	Now let $X$ be a convex body with the origin in its interior. We call such a convex body $X$ \emph{symplectically self-polar} if $X = X^{\omega}$. It is important to note that every such body is centrally symmetric~\cite[Lemma 3.1]{berezovik2022symplectic}.
	
	The main motivation for the study of symplectically self-polar convex bodies was the connection to (symmetric) Mahler's conjecture~\cite{Mahler1939}, which states that
	\[
		\vol K \cdot \vol K^\circ \geq \frac{4^n}{n!}
	\]  
	for any centrally symmetric convex body $K \subset \mathbb{R}^n$, where
	\[
		K^\circ = \{y \in \mathbb{R}^{n}\ |\ \forall x \in K\  \langle x,y \rangle \leq 1\}.
	\]	
	 In~\cite{berezovik2022symplectic} it was shown that this conjecture is equivalent to a conjecture about volumes of symplectically self-polar bodies. The latter conjecture states that
	\begin{equation}\label{eq:conj:vol_1}
		\vol X \geq \frac{2^n}{n!}
	\end{equation}
	for any symplectically self-polar convex body $X \subset \mathbb{R}^{2n}$.

	In dimension 2 it is true that $\vol X \geq 3$ for every symplectically self-polar convex body~\cite[Corollary 5.2]{berezovik2022symplectic}. So the inequality~\eqref{eq:conj:vol_1} is not tight even in dimension 2. Note that this does not contradict the tightness of Mahler's conjecture, since the proof of equivalence of these two conjectures uses tools of asymptotic geometric analysis. So the natural question is, ``What is the tight version of inequality~\eqref{eq:conj:vol_1}?''

	We may try to formulate a stronger version of this inequality using Viterbo's conjecture for centrally symmetric convex bodies. A particular case of Viterbo's conjecture~\cite{viterbo2000metric} for the Ekeland--Hofer--Zehnder capacity states that
	\[
		\vol X \geq \frac{c_{EHZ}(X)^n}{n!},
	\]
	where $X \subset \mathbb{R}^{2n}$ is a convex body and $c_{EHZ}(X)$ is the Ekeland--Hofer--Zehnder capacity of the body $X$. This conjecture fails for the class of all convex bodies (see~\cite{haim2024counterexample}), however, it is still open for the case of centrally symmetric convex bodies. Combining this conjecture with the inequality 
	\begin{equation}\label{eq:cehz_selfpol}
		c_{EHZ}(X) \geq 2 + \frac{1}{n},
	\end{equation}
	for a symplectically self-polar convex body $X \subset \mathbb{R}^{2n}$ (see \cite[Theorem 5.1]{berezovik2022symplectic}), we can write a stronger version of the inequality~\eqref{eq:conj:vol_1} as follows
	\begin{equation}\label{eq:conj:vol_2}
		\vol X \geq \frac{\left(2 + 1/n\right)^n}{n!}.
	\end{equation}

	The proof of the inequality~\eqref{eq:cehz_selfpol} is based on the following theorem:
	\begin{theorem}[Akopyan, Karasev,~\cite{akopyan2017estimating}]\label{theorem:akopyan_karasev}
	 		For a centrally symmetric convex body $X \subset \mathbb{R}^{2n}$
	 		\[
	 			c_{EHZ}(X) \geq \left(2 + \frac{1}{n}\right) c_{J}(X),
	 		\]
	 		where
	 		\[
	 			c_{J}(X)^{-1} =  \max \{|\omega(x,y)|\ | \  x,y \in X^{\omega}\}.
	 		\]
	 \end{theorem} 

	 Using also the fact that $c_{J}(X) = 1$ for a symplectically self-polar body $X$, we obtain inequality~\eqref{eq:cehz_selfpol}. A slightly weak version of Theorem~\ref{theorem:akopyan_karasev} was initially proved by Gluskin and Ostrover in~\cite[Theorem 1.6]{gluskin2016asymptotic}. The proof of Theorem~\ref{theorem:akopyan_karasev} can also be found in Appendix of~\cite{berezovik2022symplectic}.

	 The authors of~\cite{akopyan2017estimating} believed that the estimate in Theorem~\ref{theorem:akopyan_karasev} is not sharp. But in this paper we construct a sequence of symplectically self-polar polytopes $P^{\ltimes n} \subset \mathbb{R}^{2n}$ with the following property (Section~\ref{chapter:polytope_min_capacity}):

 	 \begin{theorem}\label{main:cehz}
 	 For the sequence of symplectically self-polar polytopes $P^{\ltimes n} \subset \mathbb{R}^{2n}$ one has
 	 \[
 	 c_{EHZ}(P^{\ltimes n}) = 2 + 1/n.
	\]  
  \end{theorem}

  Numerical experiments (Section~\ref{section:numerical}) suggest that inequality~\eqref{eq:conj:vol_2} is not sharp even in dimension 4, and it seems that symplectically self-polar polytopes $P^{\ltimes n} \subset \mathbb{R}^{2n}$ minimize the volume (Section~\ref{chapter:min_vol}). That is, the hypothetical sharp lower bound for the volume of symplectically self-polar bodies is the following:
  \[
  	\vol X \geq \vol P^{\ltimes n} = \frac{2^n}{n!} \cdot \frac{\Gamma\left(n+\frac{3}{4}\right)}{\Gamma\left(n+\frac{1}{2}\right)} \cdot \frac{\Gamma\left(\frac{1}{2}\right)}{\Gamma\left(\frac{3}{4}\right)}= \frac{2^n}{n!} \cdot O\left( n^{1/4} \right),
  \]
  where $\Gamma (x)$ is the gamma function (Theorem~\ref{thm:volume} and Theorem~\ref{thm:prop_P}).

  To construct the sequence of polytopes $P^{\ltimes n}$ we introduce a new operation over symplectically self-polar bodies, which we call a \emph{symplectic  $P$-suspension} (Section~\ref{chapter:increase_dim}).

  We mostly focus on symplectically self-polar polytopes in this paper. Some properties of more general self-polar polytopes were studied in~\cite{Jensen2019SelfpolarP}. 

  As shown in recent paper~\cite{berezovik2025outer}, the class of symplectically self-polar convex bodies is of interest to the researchers working on outer billiard maps.

\subsection*{Acknowledgments} The author expresses particular gratitude towards Roman Karasev for his guidance and useful comments. We thank Pazit Haim-Kislev for numerical computation of $c_{EHZ}(P^{\ltimes 2})$, which prompted the author to formulate and prove Theorem~\ref{main:cehz}; Alexander Kamal, Alexey Balitskiy, Anastasiia Sharipova, Yaron Ostrover, Michael Fraiman, and the anonymous referee for useful remarks and comments.

\section{Increasing the dimension of symplectically self-polar bodies}\label{chapter:increase_dim}

	Consider the hexagon $P = \conv\{\pm(1,1),\pm(1,0),\pm(0,1)\} \subset \mathbb{R}^2$ and fix the point $e = (1,1) \in P$ (see Figure~\ref{fig:P}). It was shown in~\cite{berezovik2022symplectic} that the hexagon $P$ has the minimal volume among all 2-dimensional symplectically self-polar bodies. 
  \begin{figure}[ht]
	\centering 
    \includegraphics[width=0.4\textwidth]{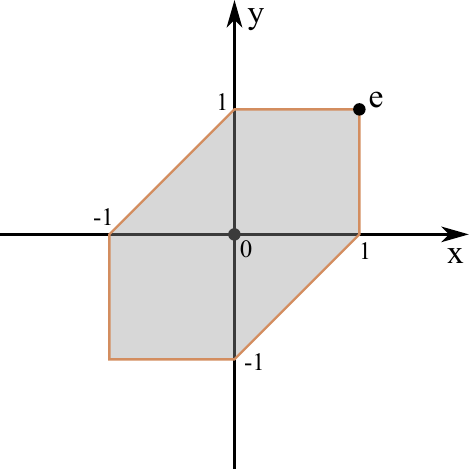}
    \caption{Hexagon $P$.}
    \label{fig:P}
  \end{figure}		
  
	\begin{definition}
		Let $X \subset \mathbb{R}^{2n}$ be a centrally symmetric convex body. The \emph{symplectic $P$-suspension of $X$} is the set
		\[
       P \ltimes_{\omega} X := \{(v_1,v_2) \in P \times \mathbb{R}^{2n}\ | \ |\omega(e,v_1)| + \|v_2\|_X \leq 1\},
    	\]
    	where $\|\cdot\|_X$ is a norm on $\mathbb{R}^{2n}$ whose unit ball is $X$.
	\end{definition}
	
	\begin{remark}
		$P \ltimes_{\omega} X$ is a subset of the symplectic product $\mathbb{R}^2 \times \mathbb{R}^{2n} \cong \mathbb{R}^{2n+2}$. Throught the paper, when we write $(a,b) \in \mathbb{R}^{2n+2} \cong \mathbb{R}^2 \times \mathbb{R}^{2n}$ or $a \oplus b \in \mathbb{R}^{2n+2}\cong \mathbb{R}^2 \times \mathbb{R}^{2n}$, we mean that $a \in \mathbb{R}^2$ and $b \in \mathbb{R}^{2n}$.
	\end{remark}
	Obviously, $P \ltimes_{\omega} X$ is a centrally symmetric convex body. Less obvious is the fact that the symplectic $P$-suspension preserves symplectic self-polarity.

	\begin{theorem}\label{main:thm}
		If $X$ is a symplectically self-polar convex body, then $P \ltimes_{\omega} X$ is also symplectically self-polar.
	\end{theorem}

	The proof of this theorem relies on the two following lemmas.
	
	\begin{lemma}\label{lem:omega:ineq}
		Let $X \subset \mathbb{R}^{2n}$ be a centrally symmetric convex body. Then
		\[
			|\omega(v,w)| \leq \|v\|_X \cdot \|w\|_{X^{\omega}}
		\]
		for all $v,w \in \mathbb{R}^{2n}$.
		
		Moreover, for any $v \in \mathbb{R}^{2n}$ there is a point $w \in \partial X^{\omega}$ such that
		\[
			|\omega(v,w)| = \|v\|_X.
		\]
	\end{lemma}	

	\begin{proof}
		For the first part of the lemma, assume that $v,w \neq 0$. The left hand side and the right hand side are homogeneous with respect to scaling of $v$ and $w$. Scale $v$ and $w$ so that $\|v\|_X = \|w\|_{X^{\omega}} = 1$. Then the statement follows from the definition of $X^\omega$.
		
		For the second part similarly assume that $\|v\|_X = 1$. Then $v \in \partial X$. From the first part it follows that $|\omega(v,w)| \leq 1$ for all $w \in X^{\omega}$. If the inequality is strict, then $|\omega(\alpha v,w)| \leq 1$ for some $\alpha > 1$ and any $w \in X^{\omega}$. This contradicts the fact that $(X^{\omega})^{\omega} = X$.
	\end{proof}

	\begin{lemma}\label{lem:P_omeg}
		Let $v,w \in P$. Let $t = |\omega(e,v)|$ and $s = |\omega(e,w)|$. Then
		\[
			|\omega(v,w)| \leq t + s - ts,
		\]
		with equality only when $v,w \in \partial P$.
	\end{lemma}
	
	\begin{proof}
		Let $v = (x_v, y_v)$ and $w = (x_w, y_w)$. Then $|y_v - x_v| = t$ and $|y_w - x_w| = s$. If necessary, we may replace $v$ by $-v$ and $w$ by $-w$, so that $y_v - x_v = t$ and $y_w - x_w = s$. Then $x_v \in [-1, 1-t]$, $x_w \in [-1, 1-s]$ and 
	\[
		\omega(v,w) = x_v y_w - y_v x_w = x_v(s+x_w) - x_w(t+x_v) = x_vs - x_wt.
	\]
	Therefore,
	\[
		- s - (1-s)t\leq \omega(v,w) \leq (1-t)s + t.
	\]
	\end{proof}
	
	\begin{proof}[Proof of Theorem~\ref{main:thm}]
		First, we show that the inequality $|\omega(v,w)| \leq 1$ holds for any points $v,w \in P \ltimes_{\omega} X$. It will prove the inclusion $P \ltimes_{\omega} X \subseteq (P \ltimes_{\omega} X)^{\omega}$. By the definition, $v = (v_1,v_2)$ and $w = (w_1, w_2)$, where $v_1,w_1 \in P$ and $v_2,w_2 \in \mathbb{R}^{2n}$. Note that
		\[
			|\omega(v,w)| = |\omega(v_1,w_1) + \omega(v_2,w_2)| \leq |\omega(v_1,w_1)| + |\omega(v_2,w_2)|.
		\]
		Let $t = |\omega(e,v_1)|$ and $s = |\omega(e,w_1)|$. Then $\|v_2\|_X \leq 1-t$ and $\|w_2\|_X \leq 1-s$. By Lemma~\ref{lem:omega:ineq} the inequality 
		\[
			|\omega(v_2,w_2)| \leq \|v_2\|_X \cdot \|w_2 \|_{X^{\omega}} = \|v_2\|_X \cdot \|w_2 \|_X \leq (1-t)(1-s)
		\]
		holds.
		And by Lemma~\ref{lem:P_omeg} the inequality 
		\[
			|\omega(v_1,w_1)| \leq t + s - ts
		\]
		also holds.
		Combining these inequalities, we obtain the desired inequality.
		
		Conversely, we show that $(P \ltimes_{\omega} X)^{\omega} \subseteq (P \ltimes_{\omega} X)$. To prove this, note first that $P \times \{0\} \subset P \ltimes_{\omega} X$, and consequently, $(P \ltimes_{\omega} X)^{\omega} \subset P \times \mathbb{R}^{2n}$. Take a point $v \in P \times \mathbb{R}^{2n} \setminus P \ltimes_{\omega} X$. This point has the representation $v = (v_1,v_2)$, where $v_1 \in P$ and $\|v_2\|_X > 1-|\omega(e,v_1)|$. By Lemma~\ref{lem:omega:ineq} there is a point $u \in X$ such that $|\omega(v_2,u)| > 1-|\omega(e,v_1)|$. Note that $(e, \pm u) \in P \ltimes_{\omega} X$, and
		\[
			\max |\omega((v_1,v_2), (e, \pm u))| = |\omega(v_1,e)| + |\omega(v_2,u)| > 1.
		\]
		Therefore, $v \notin (P \ltimes_{\omega} X)^{\omega}$. 
	\end{proof}

	\begin{remark}
		Lemma~\ref{lem:P_omeg} is crucial for the proof of Theorem~\ref{main:thm}. It does not hold if we replace $P$ by an arbitrary 2-dimensional symplectically self-polar body, for example, the Euclidean unit ball $B$, take a point $\tilde{e}$ on its border, and consider
		\[
       B \ltimes_{\omega} X := \{(v_1,v_2) \in B \times \mathbb{R}^{2n}\ | \ |\omega(\tilde{e},v_1)| + \|v_2\|_X \leq 1\}.
    	\]
    	A direct calculation shows that this body is not symplectically self-polar for any symplectically self-polar $X \subset \mathbb{R}^{2n}$. 
	\end{remark}

	The convex body $P \ltimes_{\omega} X$ may be represented in another way.
	
	\begin{theorem}\label{th:represent}
		Let $X \subset \mathbb{R}^{2n}$ be a centrally symmetric convex body. Then
		\begin{equation}\label{eq:another_form_suspension}
     		P \ltimes_{\omega} X = P \times \mathbb{R}^{2n} \cap \left(\{\pm e\} \times X^{\omega}\right)^{\omega}
	 	\end{equation}
		and
		\begin{equation}\label{eq:another_form_suspension_polar}
		    (P \ltimes_{\omega} X)^{\omega} = \conv \left(P \times \{0\} \cup \{\pm e\} \times X^{\omega}\right).
		\end{equation}
	\end{theorem}

	\begin{proof}
		By definition, $P \ltimes_{\omega} X \subset P \times \mathbb{R}^{2n}$. Let $(v_1,v_2) \in P \ltimes_{\omega} X$. We will show that $(v_1,v_2)$ belongs to the set in the right-hand side of equality~\eqref{eq:another_form_suspension}. Consider a point $u \in X^{\omega}$. Using Lemma~\ref{lem:omega:ineq} and the definition of $P \ltimes_{\omega} X$, we obtain
		\[
			|\omega((v_1,v_2),(\pm e,u))| \leq |\omega(v_1,e)| + |\omega(v_2,u)| \leq |\omega(v_1,e)| + \|v_2\|_X \leq 1.
		\]
		This proves the inclusion in one direction.
		
		Now let $(v_1,v_2) \in P \times \mathbb{R}^{2n} \cap \left(\{\pm e\} \times X^{\omega}\right)^{\omega}$. Then by Lemma~\ref{lem:omega:ineq} there is a point $u \in X^{\omega}$ such that
	\[
      	|\omega(v_1,e)| + \|v_2\|_X = |\omega(v_1,e)| + |\omega(v_2,u)| = \max|\omega((v_1,v_2),(e,\pm u))| \leq 1.
    \]
    This proves the reverse inclusion.
    
    Now we prove equality~\eqref{eq:another_form_suspension_polar}. Note that for arbitrary closed centrally symmetric nonempty convex sets $A,B,D \subseteq \mathbb{R}^{2n}$ one has $(A \cap B)^\omega = \conv(A^\omega \cup B^{\omega})$ and $(D^{\omega})^\omega = D$. Proof of these facts is similar to that in the case of standard polarity, which can be found in~\cite[Section 1.6]{Schneider_2013}.

    Let $A = P \times \mathbb{R}^{2n}$ and $B = \left(\{\pm e\} \times X^{\omega}\right)^{\omega}$. It is not hard to prove that $A^\omega = P \times \{0\}$. Since $\conv(\{\pm e\} \times X^{\omega}) = [-e,e]\times X^{\omega}$, and the symplectic polar of the convex hull is the same as symplectic polar of initial set, we have that $B = \left(\{\pm e\} \times X^{\omega}\right)^{\omega} = \left([-e,e] \times X^{\omega}\right)^{\omega}$.  Now let $D = [-e,e]\times X^{\omega}$. It is obvious that $D$ is a closed centrally symmetric convex set, therefore, $B^{\omega} = (D^{\omega})^{\omega} = D$. Thus,
    \[
    	(P \ltimes_{\omega} X)^{\omega} = \conv \left(P \times \{0\} \cup [-e,e] \times X^{\omega}\right) = \conv \left(P \times \{0\} \cup \{\pm e\} \times X^{\omega}\right).
    \] 
	\end{proof}

	\begin{corollary}\label{cor:forselfpol}
		Let $X \subset \mathbb{R}^{2n}$ be a symplectically self-polar convex body. Then
		\[
         	P \ltimes_{\omega} X = P \times \mathbb{R}^{2n} \cap \left(\{\pm e\} \times X\right)^{\omega} = \conv \left(P \times \{0\} \cup \{\pm e\} \times X\right).
	    \]
	\end{corollary}
	\begin{corollary}\label{cor:vertices}
		Let $K \subset \mathbb{R}^{2n}$ be a symplectically self-polar convex polytope and $V(K)$ be the set of its vertices. Then $P \ltimes_{\omega} K$ is a polytope and
		\[
			V(P \ltimes_{\omega} K) = (\{\pm(1,0),\pm(0,1)\} \times \{0\}) \cup (\{\pm e\} \times V(K)).
		\] 
	\end{corollary}
	\begin{proof}
		By Corollary~\ref{cor:forselfpol},
		\[
			P \ltimes_{\omega} K = \conv \left(P \times \{0\} \cup \{\pm e\} \times K\right).
		\]
		Hence,
		\[
			P \ltimes_{\omega} K = \conv  (\{\pm(1,0),\pm(0,1)\} \times \{0\} \cup \{\pm e\} \times V(K)).
		\]
		It is easy to check that the set inside the convex hull operation is in convex position.
	\end{proof}

	We can also derive how the volume changes under the symplectic $P$-suspension.
	
	\begin{theorem}\label{thm:volume}
		Let $X \subset \mathbb{R}^{2n}$ be a centrally symmetric convex body. Then
		\[
	      \vol P \ltimes_{\omega} X = \frac{4n+3}{(n+1)(2n+1)} \vol X.
    	\]
	\end{theorem}
	
	\begin{proof}
		\[
      \vol P \ltimes_{\omega} X = \int_{\mathbb{R}^2} \int_{\mathbb{R}^{2n}} \scalebox{1.5}{$\chi$}_{P \ltimes_{\omega} X}(x,y,z)\; dx\, dy\, dz, 
    	\]
    	where $\scalebox{1.5}{$\chi$}_{P \ltimes_{\omega} X}$ is the indicator function of the set $P \ltimes_{\omega} X$; $x,y$ are standard coordinates in $\mathbb{R}^2$; and $z \in \mathbb{R}^{2n}$.
    	
    	We make the change of variables $y - x = t$ and $x + y = s$.
    	\begin{multline*}
      		\vol P \ltimes_{\omega} X = \frac{1}{2}\int_{\mathbb{R}^2} \int_{\mathbb{R}^{2n}} \scalebox{1.5}{$\chi$}_{P \ltimes_{\omega} X}\left(\frac{s-t}{2},\frac{s+t}{2},z\right)\; dt\, ds\, dz =\\ 
      		= \frac{1}{2} \int_{\mathbb{R}^2} \int_{\mathbb{R}^{2n}} \scalebox{1.5}{$\chi$}_{P}\left(\frac{s-t}{2},\frac{s+t}{2}\right)\cdot  \scalebox{1.5}{$\chi$}_{(1-|t|) X}(z)\; dt\, ds\, dz = \\
      		=  \frac{\vol X}{2} \int_{\mathbb{R}^2} \scalebox{1.5}{$\chi$}_{P}\left(\frac{s-t}{2},\frac{s+t}{2}\right)  \cdot (1-|t|)^{2n}\; ds\, dt =\\
      		=  \vol X \int_{0}^{1}dt\ (1-t)^{2n} \int_{t-2}^{2-t}ds =\\
      		 = 2\vol X \int_{0}^{1} (1-t)^{2n} \cdot (2-t)\; dt =  \frac{4n+3}{(n+1)(2n+1)} \vol X. 
    	\end{multline*}
	\end{proof}

	Here we would also like to discuss the affine cylindrical capacity of a convex body, which could be defined as follows (for details, see, for example~\cite{gluskin2016asymptotic}).
	\begin{definition}
		The \emph{affine cylindrical capacity} of a convex body $X \subset \mathbb{R}^{2n}$ is defined as
		\[
			c_{ZA}(X) = \inf_{\varphi} \vol \pi \left( \varphi (X) \right),
		\]
		where the infimum is taken over all linear symplectomorphisms of $\mathbb{R}^{2n}$, and the map $\pi \colon \mathbb{R}^{2n} \to \mathbb{R}^2$ is the orthogonal projection on the first two coordinates. 
	\end{definition}

	\begin{theorem}\label{thm:affine_cap}
		For any centrally symmetric convex body  $X\subset \mathbb{R}^{2n}$, the inequality $c_{ZA}(P \ltimes_{\omega} X) \leq 3$ holds. If $X = X^{\omega}$, then $c_{ZA}(P \ltimes_{\omega} X) = 3$.
	\end{theorem}
	\begin{proof}
		Note that the orthogonal projection on the first two coordinates equals $P$, and $\vol P = 3$. The first part of the lemma follows from this.  The second part follows from~\cite[Theorem 6.3]{berezovik2022symplectic} which asserts that $c_{ZA}(X) \geq 3$ if $X^\omega \subseteq X$.
	\end{proof}

	\section{Symplectically self-polar polytopes of minimal capacity}\label{chapter:polytope_min_capacity}
  	\begin{definition}
  		We inductively define the sequence of bodies $P^{\ltimes n} = P \ltimes_{\omega} P^{\ltimes n-1}$ for any positive integer $n > 1$, where $P^{\ltimes 1} = P$.
  	\end{definition}

  	\begin{theorem}\label{thm:prop_P}
  		$P^{\ltimes n}$ is a sequence of symplectically self-polar convex polytopes
  		\begin{enumerate}
  			\item with volume
  			\[
  				\vol P^{\ltimes n} = \frac{2^n}{n!} \cdot \frac{\Gamma\left(n+\frac{3}{4}\right)}{\Gamma\left(n+\frac{1}{2}\right)} \cdot \frac{\Gamma\left(\frac{1}{2}\right)}{\Gamma\left(\frac{3}{4}\right)},
  			\]
  			(Here $\Gamma(x)$ is the gamma function)
  			\item with number of vertices
  			\[
  				|V(P^{\ltimes n})| = 10\cdot \left(2^{n-1} - 1\right) + 6,
  			\]
  			\item with affine cylindrical capacity
  			\[
  				c_{ZA}(P^{\ltimes n}) = 3.
  			\]
  		\end{enumerate}
  	\end{theorem}
  	
  	\begin{proof}
  		The fact that $P^{\ltimes n}$ is a symplectically self-polar convex polytope follows from Theorem~\ref{main:thm}, Corollary~\ref{cor:vertices}, and the fact that $P$ is a symplectically self-polar convex polytope.

  		Item (1). First, note that
  		\[
  				\frac{2^1}{1!} \cdot \frac{\Gamma\left(1+\frac{3}{4}\right)}{\Gamma\left(1 + \frac{1}{2}\right)}\cdot\frac{\Gamma\left(\frac{1}{2}\right)}{\Gamma\left(\frac{3}{4}\right)} = 2 \cdot \frac{3/4}{1/2} = 3 = \vol P.
  		\]
  		Second, the statement of the induction step follows from Theorem~\ref{thm:volume}. Indeed,
  		\begin{align*}
  				\vol P^{\ltimes n+1} &= \frac{4n+3}{(n+1)(2n+1)} \cdot \vol P^{\ltimes n} = \frac{4n+3}{(n+1)(2n+1)} \cdot \frac{2^n}{n!} \cdot \frac{\Gamma\left(n+\frac{3}{4}\right)}{\Gamma\left(n+\frac{1}{2}\right)} \cdot \frac{\Gamma\left(\frac{1}{2}\right)}{\Gamma\left(\frac{3}{4}\right)} =\\
  				&= \frac{2^{n+1}}{(n+1)!} \cdot \frac{\left(n + \frac{3}{4}\right) \cdot \Gamma\left(n+\frac{3}{4}\right)}{\left(n + \frac{1}{2}\right) \cdot \Gamma\left(n+\frac{1}{2}\right)} \cdot \frac{\Gamma\left(\frac{1}{2}\right)}{\Gamma\left(\frac{3}{4}\right)} = \frac{2^{n+1}}{(n+1)!} \cdot \frac{\Gamma\left((n+1)+\frac{3}{4}\right)}{\Gamma\left((n+1)+\frac{1}{2}\right)} \cdot \frac{\Gamma\left(\frac{1}{2}\right)}{\Gamma\left(\frac{3}{4}\right)}.
  		\end{align*}

  		Item (2). Since $|V(P)| = 6$, by Corollary~\ref{cor:vertices}
  		\[
  			|V(P^{\ltimes n+1})| = 4 + 2\cdot  |V(P^{\ltimes n})|.
  		\]
  		Thus,
  		\[
  			|V(P^{\ltimes n+1})| = 4 + 2\cdot \left(10\cdot\left(2^{n-1} - 1\right) + 6\right) = 10 \cdot \left(2^{n} - 1\right) + 6.
  		\]

  		Item (3). Formula for the capacity follows from Theorem~\ref{thm:affine_cap}.
  	\end{proof}

  	\begin{remark}
  		A direct calculation shows that the sequence of bodies $P^{\ltimes n}$ is an example of a sequence of centrally symmetric convex bodies, such that
  		\[
  			\vol (P^{\ltimes n}) < \frac{c_{ZA}(P^{\ltimes n})}{n!}
  		\]
  		for $n \geq 2$.  It follows that one cannot just prove a stronger version of centrally symmetric Viterbo's conjecture for the cylindrical affine capacity.
  	\end{remark}

  	\begin{remark}
  		It seems that polytopes $P^{\ltimes n}$ for $n \geq 2$ never appeared in literature before. Because, for example, polytope $P^{\ltimes 2}$ has 16 vertices, 44 edges, 44 two-dimensional faces, and 16 three-dimensional faces and we did not find any famous 4-polytope with the same number of faces. 
  	\end{remark}
   
	Before we start proving Theorem~\ref{main:cehz} we need to make several additional statements.
	\begin{lemma}\label{lem:induction}
		For any positive integer $n$ there is a collection of pairwise distinct vertices $v_1, \ldots, v_{2n+1}$ of the polytope $P^{\ltimes n}$, which does not contain two opposite vertices, such that
		\[
			\omega(v_i,v_j) = 1
		\]
		for all $1 \leq i < j \leq 2n+1$.
	\end{lemma}
	\begin{proof}
		We prove the theorem by induction on $n$. For the case $n = 1$, we can consider the following vertices $v_1 = (1,0),\ v_2 = (1,1),\ v_3 = (0,1)$ in $\mathbb{R}^2$. For the induction step, assume that there is such collection of vertices $v_1, \ldots, v_{2n+1}$. Consider $w_{i} = e \oplus v_i$ for $1\leq i \leq 2n+1$ and $w_{2n+2} = (0,1) \oplus 0$, $w_{2n+3} = (-1,0) \oplus 0$. From Corollary~\ref{cor:vertices} it follows that $w_1, \ldots, w_{2n+3}$ are vertices of  polytope $P^{\ltimes n+1}$. It is easy to check the other properties of this collection.
	\end{proof}
	
	\begin{remark}
		This property of the polytopes $P^{\ltimes n}$ is extreme in the following sense. In $\mathbb{R}^{2n}$ there are no $2n+2$ pairwise distinct points $v_1, \ldots v_{2n+2}$ such that $\omega(v_i,v_j) = 1$ for $i < j$. Indeed, assume the opposite. Every equation $\omega(v_i, \cdot) = 1$ defines an affine hyperplane $L_i$. Let $M_i = L_1 \cap \ldots \cap L_i$, and note that $v_{i+1} \in M_i$, but $v_{i+1} \notin L_{i+1}$; hence, the inclusions $M_1 \supset M_2 \supset \ldots \supset M_{2n}$ are strict and $M_{2n}$ contains no more than one point. It contradicts $v_{2n+1},v_{2n+2} \in M_{2n}$.
	\end{remark}	

	Consider the following theorem:
	\begin{theorem}[Haim-Kislev,~\cite{haim2019symplectic}]\label{theorem:haim}
		Let $K\subset\mathbb{R}^{2n}$ be a centrally symmetric convex polytope with a nonempty interior, $n_1, \ldots,n_{m}$ are different outward normal vectors to  one-half of its facets ($-n_1, \ldots, -n_{m}$ to the other half). Then
		\[
			\frac{1}{c_{EHZ}(K)} = \max \left\{ \sum_{1 \leq i < j \leq m} \beta_{\sigma(i)} \beta_{\sigma(j)} \omega\left(n_{\sigma(i)},n_{\sigma(j)}\right) \right\},
		\]
		where the maximum is taken over all permutations $\sigma\in S_{m}$ and all sets of reals $(\beta_i)_{i = 1}^m$ such that
		\[
        \sum_{i=1}^{m} |\beta_i| h_K(n_i) = 1.
      \]
		(Here $  	h_K(n_i) = \max \{\langle x,n_i\rangle \ |\ x \in K \} $ denotes the support function of $K$.)
	\end{theorem} 

	It is well known that for every convex polytope $K$ with an origin in its interior (not necessary centrally symmetric) there exists a one-to-one correspondence between the facets of $K$ and the vertices of $K^\circ$. Let $n$ be an outward normal vector to some facet in $K$, then the corresponding vertex is given by $n/h_{K}(n)$.  

	When $K = K^{\omega} = JK^\circ$, there is a one-to-one correspondence between the facets of $K$ and the vertices of $K$. For a given facet the corresponding vertex has the form $Jn/h_{K}(n)$, where $n$ is the outward normal vector. Using this fact and equality $\omega(Ja,Jb) =  \omega(a,b)$, we obtain 
	\begin{corollary}\label{coro:cehz_self}
		Let $K\subset\mathbb{R}^{2n}$ be a centrally symmetric convex polytope with a nonempty interior such that $K^{\omega} = K$, $v_1, \ldots, v_{m}$ is one-half of its vertices (the other half is $-v_1, \ldots, -v_{m}$). Then
		\[
			\frac{1}{c_{EHZ}(K)} = \max \left\{ \sum_{1 \leq i < j \leq m} \gamma_{\sigma(i)} \gamma_{\sigma(j)} \omega\left(v_{\sigma(i)},v_{\sigma(j)}\right) \right\},
		\]
		where the maximum is taken over all permutations $\sigma\in S_{m}$ and all sets of reals $(\gamma_i)_{i = 1}^m$ such that 
		\[
			\sum_{i=1}^{m} |\gamma_i| = 1.
		\]
	\end{corollary}

	\begin{proof}[Proof of Theorem~\ref{main:cehz}]
		We need to show that $c_{EHZ}(P^{\ltimes n}) = 2 + 1/n$. Consider the collection of vertices $v_1, \ldots, v_{2n+1}$ of the polytope $P^{\ltimes n}$ from Lemma~\ref{lem:induction}. Let $\gamma_i = 1/(2n+1)$ for $1 \leq i \leq 2n + 1$. Then
		\[
			\sum_{i=1}^{2n+1} \gamma_i = 1
		\]
		and
		\[
			\sum_{1\leq i < j \leq 2n+1} \gamma_i \gamma_j \omega(v_i,v_j) = \frac{1}{(2n+1)^2} \cdot \frac{(2n+1)^2 - (2n+1)}{2} = \frac{n}{2n+1}.
		\]
		From Corollary~\ref{coro:cehz_self} it follows that
		\[
        c_{EHZ}(P^{\ltimes n}) \leq 2 + \frac{1}{n}.
      \]
      The opposite inequality is true for every symplectically self-polar convex body~(inequality~\eqref{eq:cehz_selfpol}).
	\end{proof}

	In fact, we can deduce a more general statement about changing of capacity for symplectic $P$-suspension of symplectically self-polar polytope.
	
	\begin{theorem}
		Let $K \subset \mathbb{R}^{2n}$ be a symplectically self-polar polytope, then
		\begin{equation}\label{eq:cehz_P_razd}
        	c_{EHZ}(P \ltimes_{\omega} K) \leq 3 - \frac{1}{c_{EHZ}(K)-1}.
      	\end{equation}
	\end{theorem}
	
	\begin{proof}
		By Corollary~\ref{coro:cehz_self},
		\[
        	\frac{1}{c_{EHZ}(K)} = \sum_{1 \leq i < j \leq m} \beta_i \beta_j \omega(v_i,v_j),
      	\]
      	where $v_1, \ldots, v_{m}$ is one-half of its vertices (the other is $-v_1, \ldots, -v_{m}$) and $\{\beta_i\}_{i=1}^m$ is a collection of nonnegative numbers such that
      	\[
      		\sum_{i=1}^m \beta_i = 1.
      	\]
      	Consider the collection of points $w_i = e \oplus v_i$ for $1 \leq i \leq m$ and $w_{m+1} = (0,1) \oplus 0$, $w_{m+2} = (-1,0) \oplus 0$. By Corollary~\ref{cor:vertices} they are vertices of $P \ltimes_{\omega} K$. Moreover, this is a collection of pairwise distinct points which does not contain two opposite vertices and $\omega(w_i,w_{m+1}) = \omega(w_i,w_{m+2}) = \omega(w_{m+1},w_{m+2}) = 1$ for $1 \leq i \leq m$ and $\omega(w_i,w_j) = \omega(v_i,v_j)$ for $1\leq i,j \leq m$. Introduce the following coefficient:
      	\[
        	\alpha = \frac{c_{EHZ}(K) - 2}{3c_{EHZ}(K) - 4},
      	\]
      	and note that $0 < \alpha < 1/2$, because $c_{EHZ}(K) > 2$. Consider the collection of coefficients $\{\gamma_i\}_{i=1}^{m+2}$ defined as follows:
      	\[
        	\gamma_i = (1 - 2\alpha) \beta_i
      	\]
      	for $1 \leq i \leq m$, and
      	\[
      		\gamma_{m+1} = \gamma_{m+2} = \alpha.
      	\]
      	Then $\gamma_i \geq 0$, and
      	\[
      		\sum_{i=1}^{m+2}\gamma_i = 1.
      	\]
      	Also note that
      	\begin{align*}
        	\sum_{1 \leq i < j \leq m+2}\gamma_i \gamma_j \omega(w_i,w_j) &= (1 - 2\alpha)^2 \sum_{1 \leq i < j \leq m}\beta_i \beta_j \omega(v_i,v_j) +2 \alpha(1-2\alpha) \sum_{1 \leq i  \leq m}\beta_i + \alpha^2 =\\
       		 &= \frac{(1-2\alpha)^2}{c_{EHZ}(K)} + 2\alpha(1-2\alpha) + \alpha^2.
     	 \end{align*}
     	 Using the definition of $\alpha$, we get
     	 \[
     	 	\sum_{1 \leq i < j \leq m+2}\gamma_i \gamma_j \omega(w_i,w_j) = \frac{c_{EHZ}(K) - 1}{3c_{EHZ}(K) - 4}.
     	 \]
     	 Applying Corollary~\ref{coro:cehz_self} to the polytope $P \ltimes_{\omega} K$, we finish the proof.
	\end{proof}

	\begin{question}
		For the polytopes $P^{\ltimes n}$ inequality~\eqref{eq:cehz_P_razd} is actually an equality. Is there a symplectically self-polar polytope such that inequality~\eqref{eq:cehz_P_razd} is strict for it?
	\end{question}
	
	\section{Symplectically self-polar bodies of minimal volume}\label{chapter:min_vol}
	
	In this section we discuss the conjecture that the bodies $P^{\ltimes n}$ have the minimal volume among the symplectically self-polar convex bodies. 

	Let us start with the two-dimensional case. The fact that Ekeland--Hofer--Zehnder capacity coincides with the volume of a convex body in dimension $2$ and Theorem~\ref{theorem:akopyan_karasev} imply that $\vol(X) \geq 3$ for a symplectically self-polar convex body $X \subset \mathbb{R}^2$. Therefore, $P$ has the minimal volume among all two-dimensional symplectically self-polar convex bodies. 

	The next theorem, motivated by the referee's question, shows that $P$ is the only minimizer (up to linear symplectomorphism).

	\begin{theorem}\label{thm:2_dim_unique_min_vol}
		Let $X \subset \mathbb{R}^2$ be a symplectically self-polar convex body such that $\vol(X) = 3$. Then $X$ is linearly symplectomorphic to $P$.
	\end{theorem}

	First, we prove the following lemma.

	\begin{lemma}\label{lem:hexagon_min}
		Let $X \subset \mathbb{R}^2$ be a symplectically self-polar hexagon. Then $X$ is linearly symplectomorphic to $P$.
	\end{lemma}

	\begin{proof}
		Let $v$ be some vertex of $X$. By Lemma~\ref{lem:omega:ineq} there exists a point $w \in X$ such that $\omega(v,w) = 1$. Since $X$ is a convex polygone, the linear function $\omega(v,\cdot)$ attains its maximum on some vertex of $X$. Thus, we can assume that $w$ is a vertex of $X$. Up to an appropriate linear symplectomorphism we can assume that $v = (1,0) \in \mathbb{R}^2$ and $w = (0,1) \in \mathbb{R}^2$. Since $X$ is a centrally symmetric hexagon its vertices are $\pm v, \pm w, \pm u$ for some $u$. Without loss of generality we can assume that $u = (x_u,y_u) \in \mathbb{R}^2$ for $x_u,y_u > 0$. Since $X$ is symplectically self-polar we have $\omega(v,u) \leq 1$ and $\omega(u,w) \leq 1$, hence $x_u,y_u\leq 1$. Therefore $X \subseteq P$ and $P^\omega \subseteq X^\omega$. Thus, $X = P$, since $X = X^\omega$.
	\end{proof}

	\begin{proof}[Proof of Theorem~\ref{thm:2_dim_unique_min_vol}]
		The main ingredient of the proof of Theorem~\ref{theorem:akopyan_karasev} (see~\cite{akopyan2017estimating} or~\cite[Appendix]{berezovik2022symplectic}) is the following inequalities
		\[
				\frac{c_{EHZ}(X)}{c_{J}(X)} \geq \frac{m(X)}{2} \geq 2 + \frac{1}{n}		
		\]
		for a centrally symmetric convex body $X \subset \mathbb{R}^{2n}$ with a smooth boundary. Here $m(X)$ is the infimum of lengths of the centrally symmetric closed curves $\gamma \subseteq \partial X$ with respect to the norm $\|\cdot\|_X$. The first inequality was proven by Akopyan and Karasev in~\cite{akopyan2017estimating}. The last inequality is essentially Theorem 13F in the book~\cite{schaffer1976geometry}.

		For a centrally symmetric convex body $X$ with a smooth boundary of dimension $2$ ($n = 1$) these inequalities reduce to
		\begin{equation}\label{ineq:2_dim_min_vol}
				\frac{\vol(X)}{c_{J}(X)}  \geq \frac{l_{X}(\partial X)}{2} \geq 3,
		\end{equation}
		since the Ekeland--Hofer--Zehnder capacity coincides with the volume of a convex body, and $l_{X}(\partial X) \coloneqq  m(X)$ is the length of the boundary $\partial X$ with respect to the norm $\|\cdot\|_X$.

		Note that $\vol(X),c_{J}(X)$, and $l_{X}(\partial X)$ are continuous with respect to the Hausdorff distance on centrally symmetric convex bodies in $\mathbb{R}^2$. Indeed, $\vol(X)$ and $c_{J}(X)$ are continuous since they are 2-homogeneous with respect to scaling and monotone with respect to inclusion (the proof is similar to the proof of Theorem 1.8.20 in~\cite{Schneider_2013}). The function $l_{X}(\partial X)$ is also continuous with respect to the Hausdorff distance since it continuous with respect to the Banach--Mazur distance~\cite[Theorem 13A]{schaffer1976geometry}. Therefore, inequality~\eqref{ineq:2_dim_min_vol} holds for every centrally symmetric convex body $X \subset \mathbb{R}^2$.
		
		Now let $X \subset \mathbb{R}^2$ be a symplectically self-polar convex body such that $\vol(X) = 3$. Inequality~\eqref{ineq:2_dim_min_vol} implies that $l_{X}(\partial X) = 6$. Theorem~4K in~\cite{schaffer1976geometry} implies that $X$ is a linear image of a regular hexagon. By Lemma~\ref{lem:hexagon_min}, $X$ is linearly symplectomorphic to $P$.
	\end{proof}

	The symplectic $P$-suspension defines a map from the set of $2n$-dimensional symplectically self-polar convex bodies to the set of $(2n+2)$-dimensional symplectically self-polar convex bodies. By Theorem~\ref{thm:volume}, $\vol(P \ltimes_{\omega} X) = C_n \cdot \vol(X)$ for every symplectically self-polar convex body $X \subset \mathbb{R}^{2n}$ (the constant $C_n$ is explicitly given in Theorem~\ref{thm:volume} and depends only on $n$). Therefore, in light of Theorem~\ref{thm:2_dim_unique_min_vol}, the polytope  $P^{\ltimes 2}$  is the only symplectically self-polar convex body (up to linear symplectomorphism) of minimal volume in $\mathbb{R}^4$ that can be obtained using $P$-suspension.
	
	\begin{question}
		Does $P^{\ltimes 2}$ have the minimal volume not only among all symplectic $P$-suspensions of $2$-dimensional symplectically self-polar convex bodies, but among all $4$-dimensional symplectically self-polar convex bodies?
	\end{question}

	Suppose the answer is ``yes''.
	Would $P^{\ltimes 3}$ have the minimal volume among all six-dimensional symplectically self-polar convex bodies? And would $P^{\ltimes n}$, in general, have the minimal volume among all $2n$-dimensional symplectically self-polar convex bodies? Let us contemplate this.

	We compare this sequence with the sequence $Q^n \oplus_2 C^n$, where $Q^n = [-1,1]^n$, $C^n = (Q^n)^\circ$ is its polar cross-polytope, and $\oplus_2$ is the Lagrangian $\ell_2$-sum. These convex bodies are symplectically self-polar with
	\[
		\vol(Q^n \oplus_2 C^n) = \frac{4^n}{n!} \frac{(\Gamma(\frac{n}{2}+1))^2}{\Gamma(n+1)} = \frac{2^n}{n!}\cdot O(n^{1/2}).
	\]
	For details see~\cite{berezovik2022symplectic}. If Mahler's conjecture is true, their volumes are minimal up to a sub-exponential factor.

	This is a special case of the following construction. Let $K \subset \mathbb{R}^n$ be a centrally symmetric convex body, then the Lagrangian $\ell_2$-sum $K \oplus_{2} K^{\circ} \subset \mathbb{R}^{2n}$ is a symplectically self-polar convex body and
	\[
		\vol(K \oplus_{2} K^{\circ}) = \frac{(\Gamma(\frac{n}{2}+1))^2}{\Gamma(n+1)} \cdot \vol K \cdot \vol K^\circ.
	\]
	Note that the right-hand side of this equality coincides (up to a constant) with the expression appearing in Mahler's conjecture. Therefore, if Mahler's conjecture is true, the body $Q^n \oplus_2 C^n$ has the minimal volume among all bodies that can be obtained using construction $K \oplus_{2} K^{\circ}$. 

 	\begin{lemma}\label{lem:compare}
 		For any positive integer $n$ the following inequality holds:
 		\[
 			\frac{\vol(Q^n \oplus_2 C^n)}{\vol(P^{\ltimes n})} \geq \frac{\pi}{3}.
 		\]
 		For $n \neq 1$ this inequality is strict. Asymptotically,
 		\[
 			\frac{\vol(Q^n \oplus_2 C^n)}{\vol(P^{\ltimes n})} \sim \frac{\Gamma\left(\frac{3}{4}\right)}{\sqrt{2}} n^{1/4},
 		\]
 		as $n \to \infty$.
 	\end{lemma}

	The reader may find the proof of this lemma as well as Lemma~\ref{lem:cehz} below in Appendix~\ref{section:appendix}. 

 	Let us summarize several facts about the polytopes $P^{\ltimes n}$:
 	\begin{itemize}
 		\item They have the minimal Ekeland--Hofer--Zehnder capacity $c_{EHZ}(P^{\ltimes n}) = 2 + 1/n$ among all symplectically self-polar convex bodies.
 		\item They have the minimal affine cylindrical capacity $c_{ZA}(P^{\ltimes n}) = 3$ among all symplectically self-polar convex bodies.
 		\item The polytope $P$ has the minimal volume among all symplectically self-polar bodies in dimension $2$. The results of Section~\ref{section:numerical} with numerical experiments suggest that the polytope $P^{\ltimes 2}$ is likely to have the minimal volume in dimension $4$. With more caution, it seems that the polytope $P^{\ltimes 3}$ has the minimal volume in  dimension 6.
 		\item Everything written in the previous bullet point also applies to the number of vertices of these polytopes.
 	\end{itemize}

 	Taking into account all these arguments, let us formulate:
 	\begin{conjecture}\label{conj:1}
 		Let $X \subset \mathbb{R}^{2n}$ be a symplectically self-polar convex body. Then
 		\[
 			\vol(X) \geq \vol(P^{\ltimes n}).
 		\]
 		If $X$ is a polytope, then the number of vertices of $X$ is no less than the number of vertices of $P^{\ltimes n}$.
 	\end{conjecture}

 		The next conjecture is a stronger version of the previous one and is partially motivated by Theorem~\ref{thm:2_dim_unique_min_vol} and Lemma~\ref{lem:hexagon_min}.
 	\begin{conjecture}
 		Let $X \subset \mathbb{R}^{2n}$ be a symplectically self-polar convex body. $X$ is linearly symplectomorphic to $P^{\ltimes n}$ if one of the following items holds.
	\begin{itemize}
 		\item $X$ has the minimal volume among all $2n$-dimensional symplectically self-polar convex bodies. 
 		\item $X$ is a polytope and has the minimal number of vertices among all $2n$-dimensional symplectically self-polar convex polytopes of given dimension.
 	\end{itemize}
 	\end{conjecture}

 	How does Conjecture~\ref{conj:1} agree with Viterbo's conjecture for centrally symmetric convex bodies? Recall that the latter states that for every centrally symmetric convex body $X \subseteq \mathbb{R}^{2n}$, 
 	\[
 			\vol(X) \geq \frac{c_{EHZ}(X)^n}{n!}.
 	\]

 	Substituting $P^{\ltimes n}$ into this inequality yields

 	\begin{lemma}\label{lem:cehz}
 			The ratio
 			\[
 				\frac{n! \vol(P^{\ltimes n})}{(c_{EHZ}(P^{\ltimes n}))^n}
 			\]
 			equals $1$ for $n = 1$ and strictly increases with $n$.

 			The asymptotic formula for $n \to \infty$ has the following form:
 			\[
 				\frac{n! \vol(P^{\ltimes n})}{(c_{EHZ}(P^{\ltimes n}))^n} \sim \frac{\Gamma(1/2)}{e^{1/2}\Gamma(3/4)} n^{1/4}.	
 			\]
 	\end{lemma}

 	Thus, Conjecture~\ref{conj:1} does not contradict Viterbo's conjecture but no longer directly follows from it.
	
 	\begin{remark}
 		The upper bounds $\vol X \leq \pi^n/n!$ for the volume and $c_{EHZ}(X) \leq \pi$ for the capacity of a symplectically self-polar convex body $X \subset \mathbb{R}^{2n}$ have been obtained in~\cite{berezovik2022symplectic}. The first inequality straightforwardly follows from the Blaschke--Santal\'o inequality~\cite{Santalo1949}. Moreover, if $\vol X = \pi^n/n!$, then $X$ is linearly symplectomorphic to the unit Euclidean ball (this is a consequence of the Blaschke--Santal\'o inequality when equality is achieved). The second inequality is not so straightforward and the proof is based on Clark's variational method~\cite{clarke} and the two-dimensional Blaschke--Santal\'o inequality. It is known that the equality $c_{EHZ}(X) = \pi$ is achieved for the unit Euclidean ball. In dimension 2 it is known that the equality is achieved only for the images of the unit Euclidean ball under the action of linear symplectomorphisms since $c_{EHZ}(X) = \vol(X)$. However, in higher dimensions it is unknown whether the equality is achieved for anything other than the unit Euclidean ball (up to linear symplectomorphisms).
 	\end{remark}

	\section{Numerical experiments}\label{section:numerical}

	This section explains the theoretical basis of numerical experiments, partially explains their technical implementation, and demonstrates results. A reader interested in particular programming implementation may find the code on  GitHub\footnote{\href{https://github.com/mberezovik/Symplectically_self_polar_polytopes/}{Link} to the GitHub repository.}.
	
	To continue the discussion of algorithms let us make several simple statements.
	\begin{lemma}\label{lem:pol_check}
		Let $K \subset \mathbb{R}^{2n}$ be a centrally symmetric convex polytope. The inclusion $K \subseteq K^{\omega}$ holds if and only if for every two vertices $v,w$ of the polytope $K$, $\omega(v,w) \leq 1$.
	\end{lemma}
	\begin{proof}
		In one direction it follows from the definition of $K^{\omega}$. In the opposite direction, let $x,y \in K$. If we consider $\omega(x,y)$ as a function of the first variable ranging over $K$, then it attains a maximum on some vertex $v$. So $\omega(x,y) \leq \omega(v,y)$. Similarly, consider this function as a function of the second variable. Then there is a vertex $w$ such that $\omega(v,y) \leq \omega(v,w) \leq 1$.	
	\end{proof}

	\begin{lemma}\label{lem:increase}
		Let $K \subset \mathbb{R}^{2n}$ be a centrally symmetric convex polytope such that $K \subseteq K^{\omega}$. Let $S$ be a centrally symmetric subset of $V(K^\omega)$ such that $\omega(v,w) \leq 1$ for every $v,w \in S$. Then the polytope $	M = \conv (K \cup S) $
     	satisfies $M \subseteq M^{\omega}$.
	\end{lemma}

	\begin{proof}
		For a point $p \in K^{\omega} \setminus K$ consider $R_p = \conv \{\pm p\} \cup K$. Then $R_p \subseteq R_p^{\omega}$. Indeed, from the definition of $R_p$ it follows that every point $z \in R_p$ can be written as $z = tx \pm (1-t)p$, where $x \in K$ and $t \in [0,1]$, therefore, for any two points $z_1,z_2 \in R_p$,
		\begin{align*}
			|\omega(z_1,z_2)| = |\omega(t_1 x_1 \pm (1-t_1)p, t_2 x_2 \pm (1-t_2)p)| \leq\\
			\leq  t_1 t_2 + (1-t_1) t_2 + t_1 (1-t_2) = t_1 + t_2 - t_1t_2 \leq 1. 
		\end{align*}

		Consequently, it is sufficient to show that for $p \in S$, the inclusion $S \subset R_p^{\omega}$ holds. Let $q \in S$ and $z = tx \pm (1-t)p$, where $x \in K$ and $t \in [0,1]$. Since $q \in K^{\omega}$, then $|\omega(q,x)| \leq 1$, and since $q,p \in S$, then $|\omega(q,p)|\leq 1$, and, consequently, $|\omega(q,z)| \leq 1$. Thus, $S \subset R_p^{\omega}$.
	\end{proof}

	\subsection*{An algorithm for generating symplectically self-polar polytopes}
	
	To start with, we have a centrally symmetric convex polytope $K \subset \mathbb{R}^{2n}$ such that 
$K \subseteq K^{\omega}$.
	\begin{enumerate}
		\item Construct the polytope $K^{\omega} = JK^{\circ}$.
		\item Choose a maximal centrally symmetric vertex subset $S$ of the polytope $K^{\omega}$ satisfying assumptions of Lemma~\ref{lem:increase}.
		\item If it happens that $S$ equals to the vertex set of $K^{\omega}$, then $K^{\omega} \subseteq K$ by Lemma~\ref{lem:pol_check}. Consequently, $K^{\omega} = K$ and we halt the algorithm. Otherwise, replace the polytope $K$ by the polytope $\conv K \cup S$ and go back to the step 1.
	\end{enumerate}

	\begin{question}
		It is not clear whether this algorithm stops after finitely many steps. In practice, it turns out that it does, at least in dimension $4$. Is there a way to prove this in arbitrary dimension? If yes, then together with Zorn's lemma (see~\cite[Lemma 6.6]{berezovik2022symplectic}) this implies that the set of symplectically self-polar polytopes is dense among all symplectically self-polar convex bodies (in Hausdorff metric). 
	\end{question}

	\begin{remark}
		To construct a polar polytope $K^{\circ}$, to compute volumes of various polytopes, and to generate random polytopes the author used the software package \texttt{polymake}~\cite{polymake:2000}. All calculations were carried out using fractions, without using approximate values. Only the final volume, if necessary, was converted to an approximate decimal fraction for the convenience of interpreting the results.
	\end{remark}

 \begin{figure}[ht]
   \centering 
   \begin{subfigure}{0.45\textwidth}
     \includegraphics[width=\textwidth]{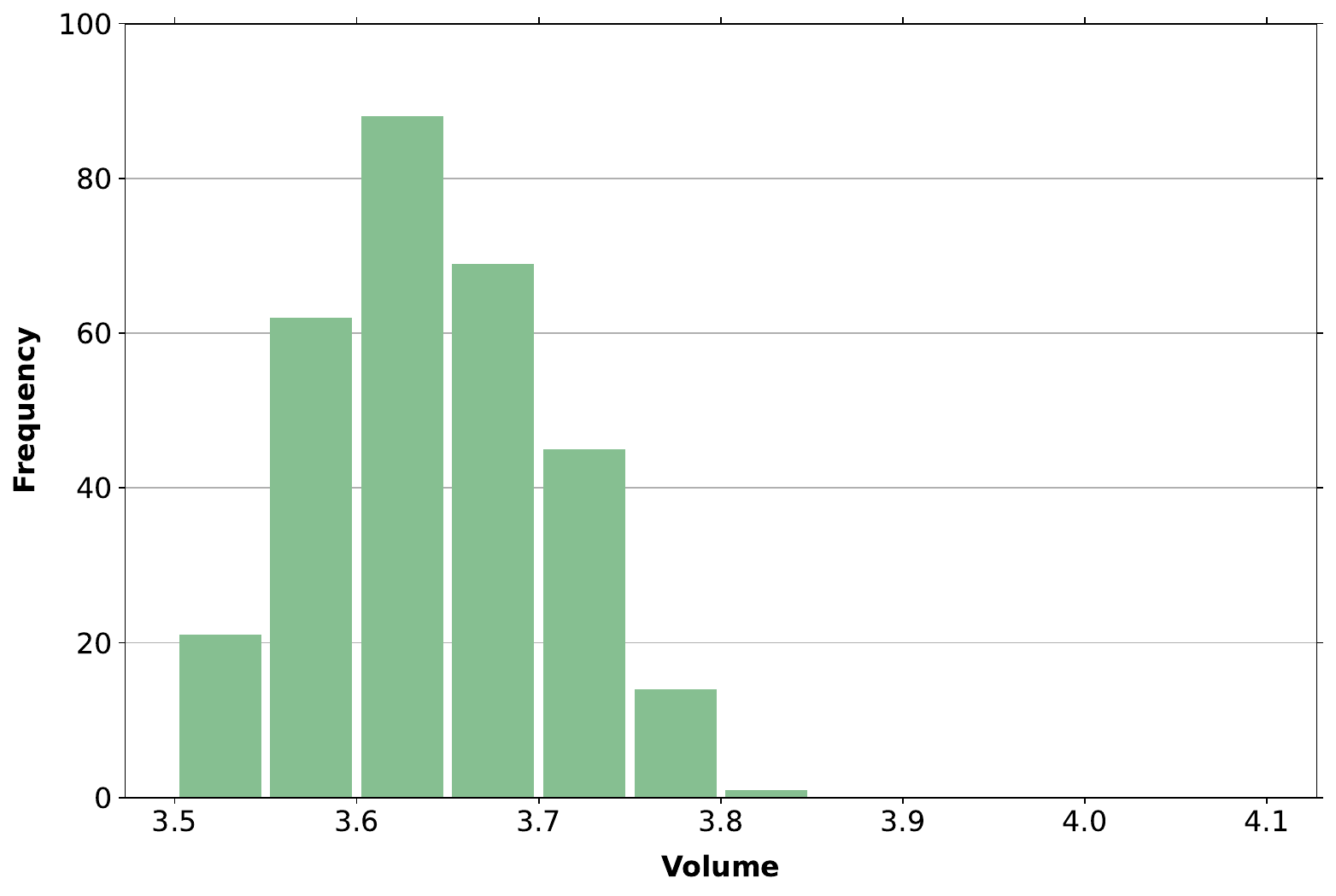}
     \caption{$k = 4$}
     \label{fig:first}
   \end{subfigure}
   \hfill
   \begin{subfigure}{0.45\textwidth}
     \includegraphics[width=\textwidth]{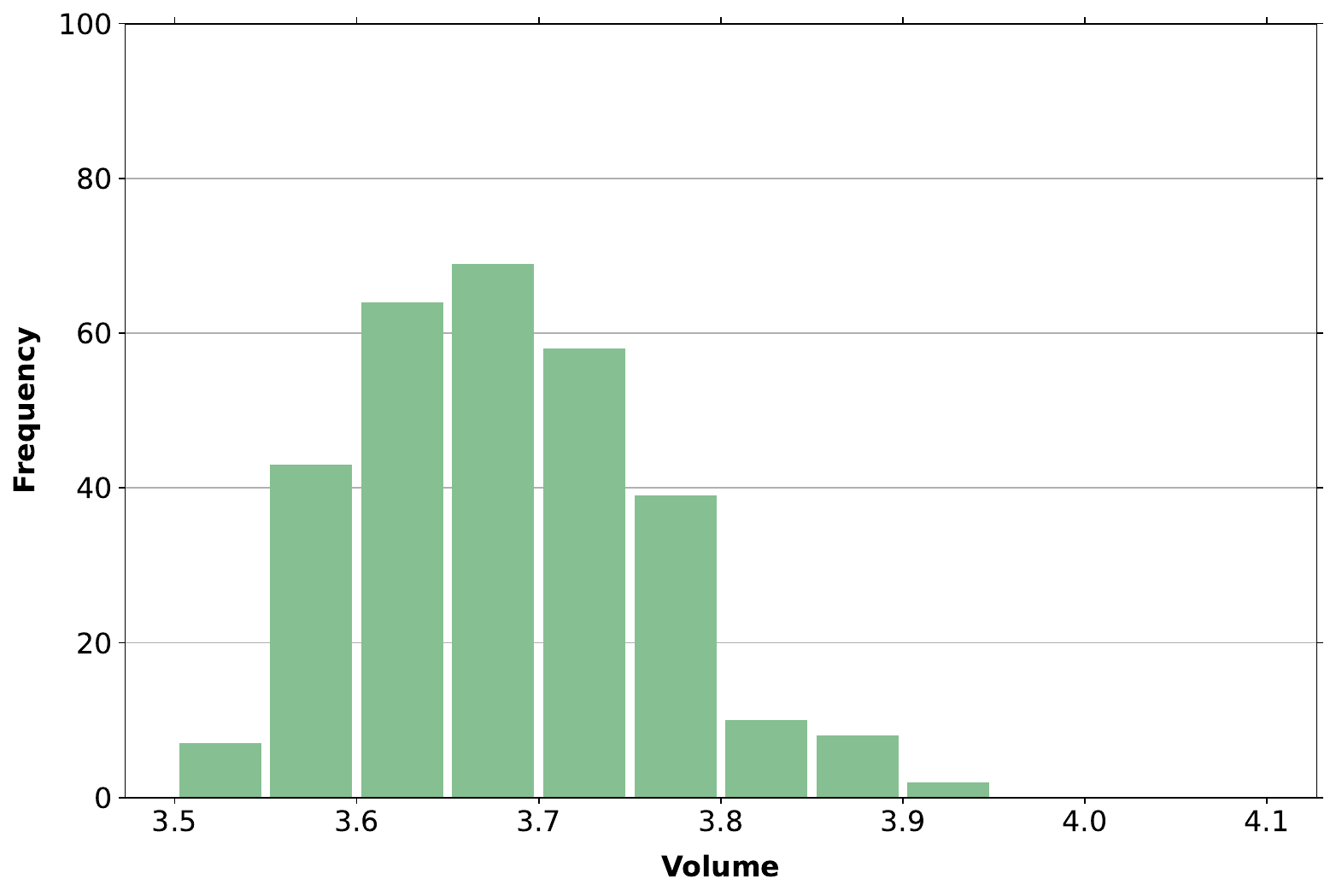}
     \caption{$k = 5$}
     \label{fig:second}
   \end{subfigure}
   \hfill
   \begin{subfigure}{0.45\textwidth}
     \includegraphics[width=\textwidth]{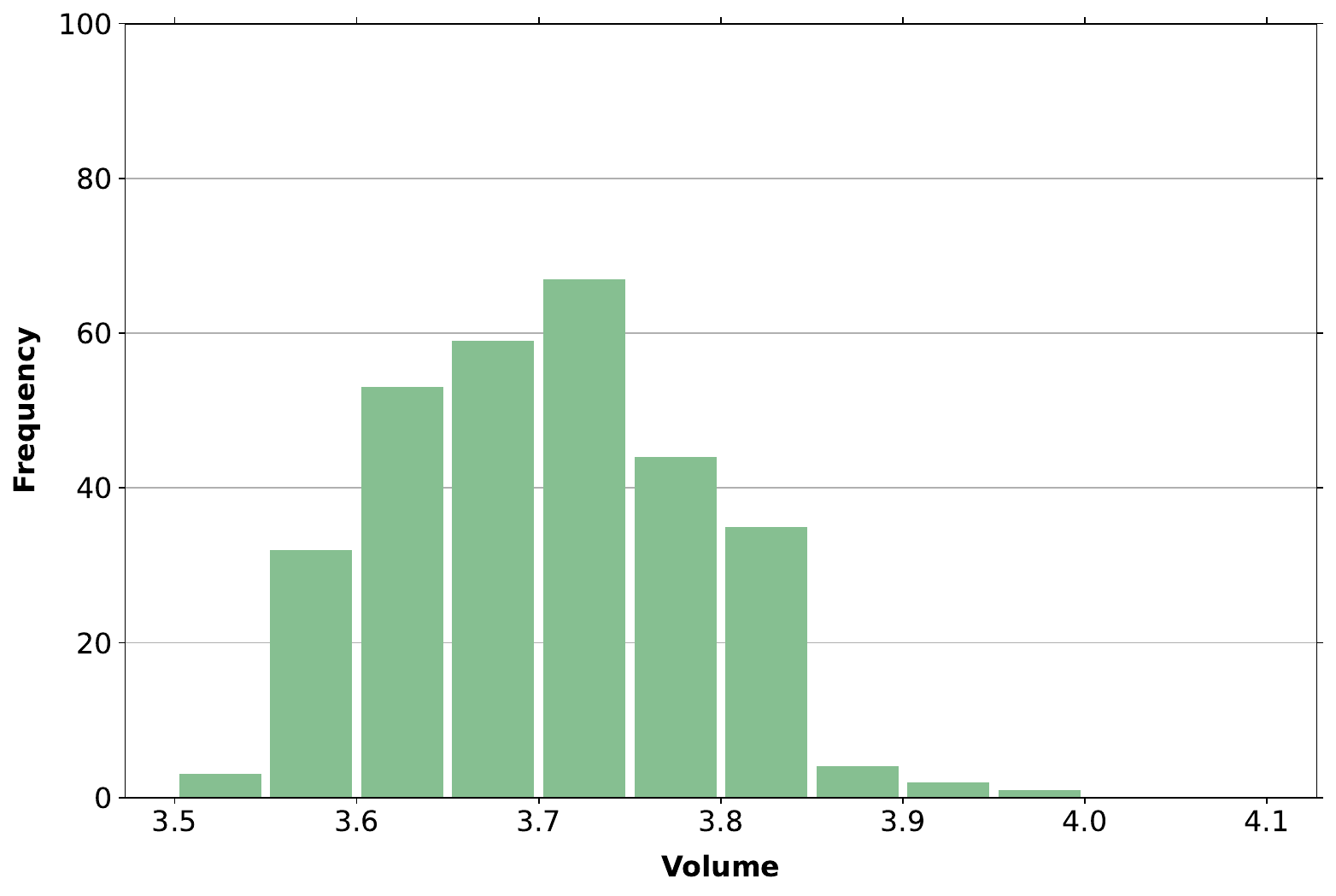}
     \caption{$k = 6$}
     \label{fig:third}
   \end{subfigure}
   \hfill
   \begin{subfigure}{0.45\textwidth}
     \includegraphics[width=\textwidth]{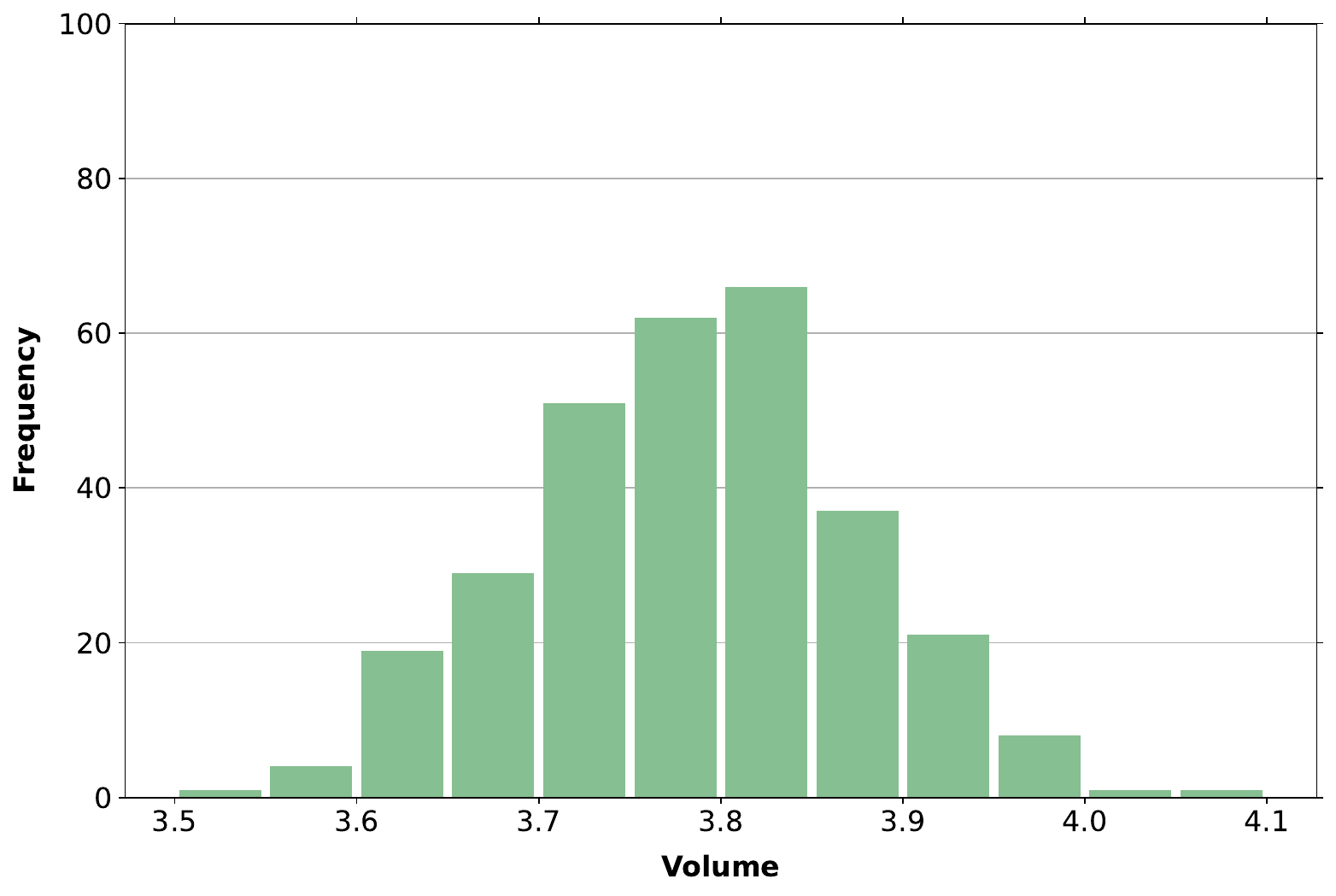}
     \caption{$k = 10$}
     \label{fig:fourth}
   \end{subfigure}
   \caption{Histograms of the volumes of the resulting polytopes for various~$k$.}
   \label{fig:figures}
 \end{figure}

	It is sufficient to take an initial polytope $K$, for the algorithm, inside the Euclidean unit ball $K \subset B^{2n}(1) = B^{2n}(1)^{\omega} \subset K^{\omega}$. For this we could generate $k \geq 2n$ random points inside the ball such that together with the origin they are in the general position.

	Using this approach in $\mathbb{R}^4$ for every fixed number of points $k \in \{4,5,6,10\}$, we generated $300$ symplectically self-polar convex polytopes and computed their volume. In Figure~\ref{fig:figures} the histogram is drawn for each $k$, where the volume of the resulting polytope is indicated on the horizontal axis.

 \begin{remark}
 		In dimension $4$, all volumes of randomly generated symplectically self-polar convex bodies were greater than $\vol(P^{\ltimes 2}) = 3.5$, and the numbers of vertices were strictly greater than $|V(P^{\ltimes 2})| = 16$.
 \end{remark}

 Note the following patterns that can be seen from these histograms. Firstly, all histograms are bounded by $3.5$ from below. Secondly, as $k$ increases, they shift to the right. Heuristically, this can be explained as follows. The greater the number $k$ the initial polytope is closer to a ball. But the Euclidean unit ball is the symplectically self-polar body with the maximal volume~\cite[Remark 1.3]{berezovik2022symplectic}.

 \begin{remark}
 	In the two-dimensional case the minimal volume is attained on a symplectically self-polar polytope with the minimal number of vertices. From this and the discussion above, it seems that the minimal volume should be attained on a symplectically self-polar polytope with the minimum possible number of vertices.
 \end{remark}

 For this reason in dimension $6$ it reasonable to take minimally possible $k = 6$. But in this dimension, with such $k$, quite serious computational difficulties appear.

\subsection*{Enumeration of all symplectically self-polar \texorpdfstring{$-1/0/1$}{TEXT} polytopes}

\begin{definition}
	A polytope $K \subset \mathbb{R}^{n}$ is called \emph{$-1/0/1$ polytope} if its vertices lie in the set $\{-1,0,1\}^{n}$.
\end{definition}

	For the enumeration of all symplectically self-polar $-1/0/1$ polytopes in dimension $2n$ the following algorithm was used:
	\begin{enumerate}
   \item  Look for an inclusion-maximal set $R \subset \{-1,0,1\}^{2n}$ such that $|\omega(v,w)| \leq 1$ for every $v,w \in R$.
   \item Construct the polytope $K_R = \conv R$. By Lemma~\ref{lem:pol_check} it turns out that $K_R \subseteq K_R^{\omega}$.
   \item Move on to the polytope $K_R^{\omega}$ and check the condition $|\omega(v,w)| \leq 1$ for its vertices. If it holds, then the polytope $K_R$ is symplectically self-polar. Otherwise, the polytope $K_R$ is not symplectically self-polar and we just ignore the set $R$.
 \end{enumerate}

 If we go through all sets $R$ from the bullet point 1, we will find all symplectically self-polar $-1/0/1$ polytopes.

 \begin{remark}
 		In dimension 4 it turns out that $K_R^{\omega} = K_R$ for all such $R$. In dimension 6 there are sets $R$ such that $K_R \subsetneq K_R^{\omega}$.
 \end{remark}

If we go through all symplectically self-polar $-1/0/1$ polytopes in dimension 4, then we can classify them by the volume and the number of vertices (see Table~\ref{Table_1}).
 \begin{table}[ht]
 \begin{center}
   \begin{tabular}{|c|c|c|c|c|}
   \hline
   $|V(K)|$ & 16   & 20   & 24   & 24   \\ \hline
   $\vol K$ & 21/6 & 22/6 & 23/6 & 24/6 \\ \hline
   \end{tabular}
  \end{center}
   \caption{All possible combinations of volumes and the number of vertices of symplectically self-polar $-1/0/1$ polytopes in dimension $4$.}
   \label{Table_1}
 \end{table}

 \begin{remark}
 	It was checked manually that every symplectically self-polar $-1/0/1$ polytope in dimension 4 with volume $3.5 = 21/6$ is linearly symplectomorphic to $P^{\ltimes 2}$.
 \end{remark}

 In dimension 6, we found 45 different combinations of number of vertices and volumes of symplectically self-polar $-1/0/1$ polytopes. There might be more as we were not able to go through all such sets $R$ because it is computationally hard. Meanwhile, among those 45, only the polytope linearly symplectomorphic to $P^{\ltimes 2}$ had the minimal volume and the minimal number of vertices.

\subsection*{Conclusions from the numerical experiments}

$P$ is a symplectically self-polar polytope with the minimal volume and the minimal number of vertices in dimension 2. In dimension 4, taking into account the results of the experiments, it seems that the polytope $P^{\ltimes 2}$ has the minimal volume and the minimal number of vertices. Note that both of them are $-1/0/1$ polytopes. So in dimension 6 it seems reasonable to search polytopes of minimal volume among $-1/0/1$ polytopes. During the search the polytope linearly symplectomorphic to $P^{\ltimes 2}$ was found, while the others had a greater volume and greater number of vertices.

\section{Appendix}\label{section:appendix}

\begin{proof}[Proof of Lemma~\ref{lem:compare}]
	Denote
	\[
		a_n = \frac{\vol(Q^n \oplus_2 C^n)}{\vol(P^{\ltimes n})} = \frac{2^n \left(\Gamma\left(\frac{n}{2} + 1\right)\right)^2 \Gamma\left(n + \frac{1}{2}\right) \Gamma\left(\frac{3}{4}\right)}{\Gamma(n+1) \Gamma \left(n+\frac{3}{4}\right) \Gamma\left(\frac{1}{2}\right)}.
	\]
	Then
	\[
		\frac{a_{n+2}}{a_n} = \frac{4 \left(\frac{n}{2} + 1\right)^2 \left(n + \frac{1}{2}\right)\left(n + \frac{3}{2}\right)}{\left(n + 1\right) \left(n + 2\right)\left(n + \frac{3}{4}\right) \left(n + \frac{7}{4}\right)}  = \frac{16n^3 + 64n^2 + 76n + 24}{16n^3 + 56n^2 + 61n + 21} > 1.
	\]
	At the same time
	\begin{align*}
     a_1 &= \frac{\pi}{3},\\
     a_2 &= \frac{8}{7}.
  \end{align*}
  It is easy to check that $a_2 > a_1$, hence $a_{n} > a_1$ for all $n>1$.

  The asymptotic formula follows directly from the Stirling approximation.
\end{proof}

\begin{proof}[Proof of Lemma~\ref{lem:cehz}]
	Denote
	\[
		a_n = \frac{n! \vol(P^{\ltimes n})}{(2+1/n)^n} = \left(\frac{2n}{2n+1}\right)^n \cdot \frac{\Gamma\left(n+\frac{3}{4}\right)}{\Gamma\left(n+\frac{1}{2}\right)} \cdot \frac{\Gamma\left(\frac{1}{2}\right)}{\Gamma\left(\frac{3}{4}\right)}.
	\]
	Then
	\[
		\frac{a_{n+1}}{a_n} = \left(\frac{2n+1}{2n}\right)^{n} \cdot \left(\frac{2n+2}{2n+3}\right)^{n+1} \frac{(n+3/4)}{(n+1/2)} = \left(\frac{2n^2 + 3n +1}{2n^2+ 3n}\right)^n \frac{(n+1)(4n+3)}{(2n+3)(2n+1)}.
	\]
	By the Bernoulli inequality,
	\[
		\left(\frac{2n^2 + 3n +1}{2n^2+ 3n}\right)^n = \left(1 + \frac{1}{2n^2 + 3n}\right)^n \geq 1 + \frac{1}{2n+3} = \frac{2n+4}{2n+3}.
	\]
	\[
		\frac{a_{n+1}}{a_n} \geq \frac{(2n+4)(n+1)(4n+3)}{(2n+3)^2(2n+1)} = \frac{8n^3 +30n^2 +34n+12}{8n^3 + 28 n^2 + 30n + 9} > 1.
	\]
	The asymptotic formula follows directly from the Stirling approximation.
\end{proof}

\bibliography{bibliography}
\bibliographystyle{abbrv}
\end{document}